\documentstyle[amssymb,amsfonts,12pt]{amsart}

\newtheorem{theorem}{Theorem}[section]

\newtheorem{proposition}[theorem]{Proposition}

\newtheorem{conjecture}[theorem]{Conjecture}

\theoremstyle{remark}

\newtheorem*{ack*}{Acknowledgment}

\def\R{{\mathbb R}}
\def\E{{\mathcal E}}

\def\C{{\mathbb C}}

\def\P{{\mathbb P}}\def\nint{\mathop{\diagup\kern-13.0pt\int}}

\def\l{{\;\lessapprox\;}}

\def\bas{\begin{align*}}
\def\eas{\end{align*}}
\def\bi{\begin{itemize}}
\def\ei{\end{itemize}}
\newenvironment{proof}{\noindent {\bf Proof} }{\endprf\par}
\def \endprf{\hfill  {\vrule height6pt width6pt depth0pt}\medskip}
\def\emph#1{{\it #1}}

\begin{document}
\author{Ciprian Demeter}
\address{Department of Mathematics, Indiana University,  Bloomington IN}
\email{demeterc@@indiana.edu}
\thanks{The author is partially supported  by the NSF Grant DMS-1161752}
\title[On the restriction theorem for paraboloid in $\R^4$]{On the restriction theorem for paraboloid in $\R^4$}

\begin{abstract}
We prove that the recent breaking \cite{Za} of the $\frac32$ barrier in Wolff's estimate on the Kakeya maximal operator in $\R^4$  leads to improving the $\frac{14}{5}$ threshold for the  restriction problem for the paraboloid in $\R^4$. One of the ingredients is a slight refinement of a trilinear estimate from \cite{G2}. The proofs are deliberately presented in a nontechnical and concise format, so as to make the arguments more readable and focus attention on the key tools.
\end{abstract}
\maketitle

\section{Kakeya and restriction estimates}
A Kakeya set in $\R^n$ is a set containing a unit line segment in every direction. The following stands one of the most fascinating conjectures in geometric measure theory.
\begin{conjecture}[Kakeya set conjecture]
Each Kakeya set in $\R^n$ has Hausdorff dimension $n$.
\end{conjecture}
For quantities $A,B$ that depend on a scale parameter $P$ (typically the radius $R$ or eccentricity $\delta^{-1}$), we will write $A\l B$ to denote the fact that $A\le C_\epsilon P^\epsilon B$ holds for all $\epsilon>0$.
An $(N_1,N_2)$-tube is a long cylinder with radius $N_1$ and length $N_2$. Its eccentricity is $N_2N_1^{-1}$.
The Kakeya set conjecture is known to be a consequence of the following conjecture.
\begin{conjecture}[Kakeya maximal operator conjecture]
Let $\Omega$ be a collection of tubes in $\R^n$ with eccentricity $\delta^{-1}$, equal sizes and $\delta$-separated directions (in particular, there is at most one tube in each of the $\sim \delta^{1-n}$ directions). Then for $\frac{n}{n-1}\le r\le \infty$
\begin{equation}
\label{3}\|\sum_{T\in\Omega}1_T\|_{r}\l (\sum_{T\in\Omega}|T|)^{\frac{1}{r}}\delta^{\frac{n}{r}-(n-1)}.
\end{equation}
\end{conjecture}

This latter conjecture is in fact a theorem when $n=2$ but is  open in higher dimensions. When $n\ge 3$, it has been verified by Wolff  \cite{Wo} for $r\ge \frac{n+2}{n}$, and improvements in high dimensions have been obtained by Katz and Tao in \cite{KT1}. Very recently, Zahl improved Wolff's result to $r\ge \frac{85}{57}$ when $n=4$.

We point out that in general, an  inequality of the form $$\|\sum_{T\in\Omega}1_T\|_{r}\l (\sum_{T\in\Omega}|T|)^{\frac{1}{r}}\delta^{-s}$$ implies that the Hausdorff dimension of Kakeya sets is at least $n-sr'$, see \cite{Tao1}. 
\bigskip

Let $$\P^{n-1}=\{(\xi_1,\ldots,\xi_{n-1},\xi_1^2+\ldots+\xi_{n-1}^2):\;|\xi_i|\le 1\}$$
denote the truncated elliptic paraboloid in $\R^n$. For a cube $\tau\subset [-1,1]^{n-1}$, $f_\tau$ will typically denote the restriction $f1_\tau$ of $f$ to $\tau$. Given $f:[-1,1]^{n-1}\to \C$, denote by
$$Ef(x_1,\ldots,x_{n})=$$$$\int f(\xi_1,\ldots,\xi_{n-1})e(\xi_1x_1+\ldots+\xi_{n-1}x_{n-1}+(\xi_1^2+\ldots+\xi_{n-1}^2)x_n)d\xi_1\ldots d\xi_{n-1}$$
the extension operator.
\medskip

Recall also the Restriction conjecture for the paraboloid.
\begin{conjecture}[Restriction conjecture]
For each $p>\frac{2n}{n-1}$ and each $f:[-1,1]^{n-1}\to \C$
\begin{equation}
\label{2}
\|Ef\|_{L^p(\R^{n})}\lesssim \|f\|_p.
\end{equation}
\end{conjecture}

A standard randomization argument shows that the validity of \eqref{2} for some $p$ implies the following weaker form of \eqref{3}
\begin{equation}
\label{3'}\|\sum_{T\in\Omega}1_T\|_{r}\l (\sum_{T\in\Omega}|T|)^{\frac{1}{r}}\delta^{\frac{2n}{r}-2(n-1)},
\end{equation}
with $r=\frac{p}{2}$.
As observed earlier, this in turn implies that the Hausdorff dimension of  Kakeya sets in $\R^n$ is at least $$d_{p,n}:=\frac{2p-n(p-2)}{p-2}.$$ In particular, since $d_{\frac{2n}{n-1},n}=n$, the Restriction conjecture is stronger than the Kakeya set conjecture.  In fact \eqref{3'} shows that it is also stronger than the Kakeya maximal operator conjecture.

The Restriction conjecture is also known when $n=2$ and open in all other dimensions. In  dimensions three and higher than four, the best known restriction estimates are weaker than the best known Kakeya estimates. This means that the Hausdorff dimension of Kakeya sets in $\R^n$ is known to be strictly larger than $d_{p,n}$, where $p$ is the smallest value for which \eqref{2} is known to hold. Interestingly, when $n=4$, recent advances due to Guth  have allowed for the restriction theory to catch up with Wolff's result for the Kakeya set conjecture. Indeed, it is proved in \cite{G2} that \eqref{2} holds with $p=\frac{14}{5}$ when $n=4$ and note that $3=d_{\frac{14}{5},4}$.
The main goal of this note is to show that any improvement over Wolff's exponent $r=\frac32$ in \eqref{3}  leads to improvements over the restriction index $\frac{14}{5}$, too.

More precisely, we will prove the following result. 

\begin{theorem}
\label{7}
Let $n=4$.
If \eqref{3} holds for some $r<\frac32$ then
$$\|Ef\|_{L^p(B_R)}\l \|f\|_\infty$$
 holds for some $p<\frac{14}5$ and each ball $B_R$ with radius $R$.
\end{theorem}

The dependence of $p$ on $r$ can be extracted from the argument. 
Using known arguments, $\|f\|_\infty$ may be replaced with $\|f\|_{p}$ and $B_R$ may be replaced with $\R^n$. In particular, combining Theorem \ref{7} with the new result \cite{Za} on the Kakeya maximal function leads to a slight improvement of the restriction index, $p=\frac{14}{5}-\frac{2}{416515}$.

The proof of the theorem will be presented in sections \ref{s2} and \ref{s1} and will involve  a slight reshuffling of the techniques from \cite{G1}, \cite{G2} and \cite{BG}.
Our hypothesis on $r<\frac{3}{2}$ will be used twice in the argument. First, a corollary of this (inequality \eqref{4}) is used in Section \ref{s2} to get a new  trilinear restriction estimate. Second, the full strength of the hypothesis is used in Section \ref{s1} to bridge the gap between the trilinear and the desired linear restriction estimate.

This material is based upon work supported by the National Science Foundation under Grant No. 1440140, while the author was in residence at the Mathematical Sciences Research Institute in Berkeley, California, during the Spring semester of 2017. The author has benefitted from discussions with Larry Guth, Marina Iliopoulou and   Alex  Iosevich. He is particularly indebted to Josh Zahl for a careful reading of this manuscript, and for sharing an early version of \cite{Za}.

\section{Tangent tubes}

Let $Z$ be an $m$-dimensional variety in $\R^n$. The polynomial method developed in \cite{G1} and \cite{G2} introduces the concept of a tube tangent  to the variety $Z$. For all practical purposes we may think of this as being an $(R^{\frac12},R)$-tube that is contained in the $CR^{\frac12}$-neighborhood of $Z$.\footnote{There is a small lie here, the analysis in \cite{G1} and \cite{G2} introduces a small parameter $\delta$ and works with $R^\delta$ enlargements of both the tubes and the wall. This complication is of entirely technical nature and will be ignored here} We will call this {\em the wall}, and will denote it by $W_{Z,R}$.

It seems very intuitive to conjecture the following, see Conjecture 11.1 from \cite{G2}.

\begin{conjecture}
\label{44}
Let $\Omega$ be a collection of $(R^{\frac12},R)$-tubes in $B_R\subset \R^n$ with $R^{-\frac12}$-separated directions. Assume these tubes are tangent to some $n-1$ dimensional variety $Z$ of degree at most $D$.
Then the number of tubes in $\Omega$ satisfies
$$\sharp\Omega\l D^{C}R^{\frac{n-2}2},$$
for some $C$ independent of $D,R$.
\end{conjecture}
To put things into perspective, we show the connection between this and the Kakeya maximal operator conjecture.
\begin{proposition}
\label{30}
Let $\Omega$ be a collection of tubes as in Conjecture \ref{44}.
Inequality \eqref{3} for some $r$ implies the bound
\begin{equation}
\label{4}
\sharp\Omega\l D^CR^{n-1-\frac{r'}{2}}.
\end{equation}
In particular, the Kakeya maximal operator conjecture implies Conjecture \ref{44}.
\end{proposition}
\begin{proof}
The proof is an immediate application of H\"older's inequality and Wongkew's inequality \cite{Won},  which states that the volume of $W_{Z,R}\cap B_R$ is $\l D^CR^{n-\frac12}$.
Indeed
$$\sharp\Omega R^{\frac{n+1}{2}}=\sum_{T\in\Omega}|T|=\|\sum_{T\in\Omega}1_T\|_1\le \|\sum_{T\in\Omega}1_T\|_r|W_{Z,R}\cap B_R|^{\frac1{r'}}$$$$\l D^C(\sum_{T\in\Omega}|T|)^{\frac1r}R^{\frac{n-1}{2}-\frac{n}{2r}+\frac{n-\frac12}{r'}},$$
and this is easily seen to imply \eqref{4}.

\end{proof}

Conjecture \ref{44} has been verified by Guth \cite{G1} when $n=3$. When $n=4$, Zahl \cite{Za} proved a slightly weaker version of the conjecture with  $D^C$ replaced by an unspecified constant $C_D$.  
The validity of the conjecture in higher dimensions is unknown.

\section{An improved trilinear restriction theorem in $\R^4$}
\label{s2}
In the following, we will restrict attention to $n=4$. For $\xi=(\xi_1,\xi_2,\xi_3)\in [-1,1]^3$, define the normal vector to $\P^3$
$$n(\xi)=(-2\xi_1,-2\xi_2,-2\xi_3,1).$$

Fix three cubes $\tau_1,\tau_2,\tau_3\subset [-1,1]^{3}$ with side length $\sim 1$. We will assume the transversality condition
$$\inf_{\xi^i\in \tau_i}|n(\xi^1)\wedge n(\xi^2)\wedge n(\xi^3)|\gtrsim 1.$$

A very close version of the following result is proved in \cite{G2}. \footnote{The slight lie here is that in \cite{G2} the minimum is taken over  a larger number of contributions than just three, by still maintaining a trilinear profile. More precisely, the term  $\|\min_{i=1}^3|Ef_{\tau_i}|\|_{L^{p}(B_R)}$ here is a substitute for the quantity $\|Ef\|_{\text{BL}_{3,A}^p(B_R)}$ from \cite{G2} that we do not bother to define. The distinction between three and the higher number considered in \cite{G2} is irrelevant for our analysis.}
\begin{theorem}
For each $f:\cup \tau_i\to \C$ and $R\ge 1$ we have for $p\ge \frac{14}{5}$
$$\|\min_{i=1}^3|Ef_{\tau_i}|\|_{L^{p}(B_R)}\l \|f\|_2.$$
\end{theorem}
It is conjectured that the minimum can be replaced with the average $(\prod_{i=1}^3|Ef_{\tau_i}|)^{\frac13}$, but this  stronger result would not help improve the argument presented here.\footnote{It would be of independent interest to determine whether the polynomial method can be used to make progress on this trilinear restriction conjecture regarding geometric averages. In its current formulation, the polynomial method does not control well the interactions between tangent and transverse tubes that are inherent to geometric averages. The choice of a substitute norm in \cite{G2} is precisely made to avoid such interactions.}

The exponent $\frac{14}{5}$ is sharp, if the $L^2$ norm of $f$ is used on the right hand side. We will show how to lower the exponent $\frac{14}{5}$ by replacing the $L^2$ norm with the $L^\infty$ norm.

\begin{theorem}
\label{9}
Assume \eqref{4} holds for some $r<\frac32$. Then there is $q<\frac{14}{5}$ such that for each $f:\cup \tau_i\to \C$ and $R\ge 1$ we have
\begin{equation}
\label{euyrycur90fi09r}
\|\min_i|Ef_{\tau_i}|\|_{L^{q}(B_R)}\l \|f\|_\infty.
\end{equation}
\end{theorem}
\begin{proof}
 We will prove the following  slightly stronger result.

Assuming that $f$ satisfies
\begin{equation}
\label{6}
\int_\theta|f|^2\lesssim |\theta|,\;\;\text{for each }R^{-1/2}-\text{cube }\theta\subset [-1,1]^3,
\end{equation}
we will show that
\begin{equation}
\label{22}
\|\min_i|Ef_{\tau_i}|\|_{L^{q}(B_R)}^q\l\|f\|_2^{\frac83}.
\end{equation}
It is clear that this implies \eqref{euyrycur90fi09r}.

The proof of \eqref{22} follows very closely the approach in \cite{G1}, with the input \eqref{5} from \cite{G2}. We briefly sketch it and refer the reader to \cite{G1} for details. 

There is a double induction on $R$ and $\|f\|_2$.  Use a polynomial $P$ of appropriate degree $D$ to create $\sim D^4$ cells. The degree $D$ is chosen to depend on $R$, but can be thought of as $\l1$. Call $Z$ the zero set of $P$, and let $W_{Z,R}$ be the corresponding wall. 

One needs to estimate the cellular contribution and the contribution from the wall. The cellular contribution is controlled via the induction on $\|f\|_2$.

Roughly speaking, on the wall one has a decomposition of the form
$$Ef_{\tau_i}=Ef_{tang,\tau_i}+Ef_{trans,\tau_i},$$
with $Ef_{tang,\tau_i}$ supported on $(R^{1/2},R)$-tubes tangent to $Z$, and $Ef_{tang,\tau_i}$ supported on $(R^{1/2},R)$-tubes that intersect the variety in a transverse (non tangential) way. The transverse contribution for the wall is controlled via the induction on $R$.
To address the tangent term contribution to the wall, it will suffice to prove
$$\|\min_i|Ef_{tang,\tau_i}|\|_{L^{q}(B_R)}\l\|f\|_2^{\frac8{3q}}.$$
This is the only new estimate, and here is how it follows.
By Proposition 8.1 from \cite{G2} ($n=4, m=k=3$), we have for $2\le q\le \frac{14}{5}$

\begin{equation}
\label{5}
\|\min_i|Ef_{tang,\tau_i}|\|_{L^{q}(B_R)}\l R^{\frac12-\frac72(\frac12-\frac{1}{q})}\|f\|_2.
\end{equation}

Using \eqref{4} and \eqref{6} we get
$$\|f\|_2\l (R^{3-\frac{r'}{2}}R^{-\frac32})^{\frac12}=R^{-t},$$
with $t>0$. This is the place where we use the fact that the dependence in \eqref{4} is polynomial in $D$, as $D\l 1$ guarantees $D^C\l 1.$
Thus, for $\frac83<q<\frac{14}{5}$, \eqref{5} can be  dominated by
$$R^{\frac12-\frac72(\frac12-\frac{1}{q})}R^{-t(1-\frac{8}{3q})}\|f\|_2^{\frac{8}{3q}}.$$
It suffices to choose $q$ sufficiently close to $\frac{14}{5}$ so that the exponent of $R$ is $\le 0$.

\end{proof}

The argument above shows that we may take $q=\frac{2(9+4r')}{3(r'+2)}$. In particular, using $r=\frac{85}{57}$ as in \cite{Za}, gives $q=\frac{8}{3}\times\frac{148}{141}$. If we assume \eqref{4} holds for $r=\frac{4}{3}$, then the corresponding value is  $q=\frac{25}{9}$.

\section{The proof of Theorem \ref{7}}
\label{s1}

There are two types of mechanisms introduced in \cite{BG} that allow to convert multilinear estimates into linear ones. The reader can check that the more basic one does not suffice for our purposes, as the treatment of the planar contribution\footnote{The terms $Ef_\tau$ with $\tau$ intersecting a line in $\R^3$} turns out to be too costly\footnote{In short, while Theorem \ref{4} gives a favorable estimate below $\frac{14}{5}$ for the trilinear term, there is no obvious way to duplicate this estimate for the bilinear term}. The more elaborate mechanism minimizes the cost for the planar term by using Kakeya type estimates.  The proof in this section follows very closely  this more elaborate approach.

We will use the following version of inequality (3.4)-(3.5) from \cite{BG}  (see also Lemma 4.3.1 from \cite{BS}), valid for $x\in B_R$
$$|Ef(x)|\l $$
\begin{equation}
\label{8}
\sum_{R^{-1/2}\lesssim \delta\lesssim 1}\max_{\tilde{\E}_\delta}\;[\sum_{\tau\in \E_\delta}(\phi_\tau(x)\min_i|Ef_{\tau_i}(x)|)^2]^{1/2}
\end{equation}
\begin{equation}
\label{99}
+\max_{{\E}_{R^{-1/2}}}[\sum_{\tau\in \E_{R^{-1/2}}}(\phi_\tau(x)|Ef_{\tau}(x)|)^2]^{1/2}
\end{equation}
where

$$(A1)\;\;\;\;\E_\delta\text{ is an arbitrary collection consisting of }O(\delta^{-1})\;\;\delta-\text{cubes } \tau $$
$$(A2)\;\;\;\;\text{for each }\E_\delta\text{ as above, }\tilde{\E}_\delta\text{ is any collection of the form}$$
$$\tilde{\E}_\delta=\{\tilde{\tau}:=(\tau,\tau_1,\tau_2,\tau_3):\;\tau\in\E_\delta\}$$
$$\text{ where }\tau_1,\tau_2,\tau_3\subset \tau \text{ are arbitrary }\frac{\delta}K-\text{cubes satisfying the non collinearity assumption }$$
$$\inf_{\xi^i\in \tau_i}|n(\xi^1)\wedge n(\xi^2)\wedge n(\xi^3)|\gtrsim \delta^2.$$
$$(A3)\;\;\;\;\phi_\tau\ge 0\text{ and }\frac1{|B|}\int_B\phi_{\tau}^4\l 1,\;\text{ for each } (\delta^{-1},\delta^{-2})-\text {tube  }B\text{ dual to }\tau.$$
Here $K$ is a large enough parameter satisfying $K\l 1$. The idea behind such a decomposition is to iterate the following dichotomy. Either there are three transverse cubes that contribute significantly, or all such cubes cluster near a line in $\R^3$, in which case one uses the standard $L^4$ Cordoba type estimate.
\bigskip

To prove Theorem \ref{7} we may assume that $\|f\|_\infty=1.$ It suffices to show that there exists  $p<\frac{14}{5}$ such that $\|\eqref{8}\|_{L^p(B_R)}\l 1$ and $\|\eqref{99}\|_{L^p(B_R)}\l 1$. We will show this for the term \eqref{8}, the analysis for the other term is entirely similar.

Let $q$ be the number from Theorem \ref{9}. Parabolic rescaling shows that for each $\tau_i$ as in $(A2)$ and each $s\ge q$
\begin{equation}
\label{10}
\|\min_i|Ef_{\tau_i}|\|_{L^{s}(B_R)}^s\l \delta^{3s-5}.
\end{equation}

We will get three estimates for \eqref{8} that we will then interpolate using H\"older. To describe these estimates, let
 $$f_1(z)=\frac{3}{2}-4z,\;\;\;\;f_2(z)=\frac52-7z.$$
 The first inequality will be
 $$\|\eqref{8}\|_{L^q(B_R)}\l \delta^{f_2(\frac1q)}$$
 and will follow from the new  trilinear estimate in Theorem \ref{9}. The advantage of this inequality is that it holds at $q<\frac{14}{5}$, while its deficit comes from the fact that $f_2(\frac1q)<0$. We will compensate this deficit by proving an estimate of the form\footnote{There will be certain losses involving truncation parameters $\lambda$ and $\mu$, but these will be balanced with a third inequality}
 $$\|\eqref{8}\|_{L^{2r}(B_R)}\l \delta^{f_1(\frac1{2r})}$$
with $\frac{14}{5}<2r<3$ as in Theorem \ref{7}. The strength of this estimate comes from the fact that $f_1(\frac1{2r})>f_2(\frac1{2r})>0$. These inequalities combined with the fact that $f_2(z)>0$ for $z<\frac{5}{14}$ will be enough to prove Theorem \ref{7}.

\medskip

Here is how to get the first estimate. Let $q\le s\le 4$. Write first using H\"older
\begin{equation}
\label{11}
\max_{\tilde{\E}_\delta}[\sum_{\tau\in \E_\delta}(\phi_\tau(x)\min_i|Ef_{\tau_i}(x)|)^2]^{1/2}\lesssim
\delta^{\frac1s-\frac12}[\sum_{\tau}(\phi_\tau(x)\min_i|Ef_{\tau_i}(x)|)^s]^{1/s}.
\end{equation}
Note that the sum on the right is over all cubes $\tau$ in a  partition of $[-1,1]^3$. Consider a finitely overlapping cover of $B_R$ with $(\delta^{-1},\delta^{-2})$-tubes $B$ dual to $\tau.$ Since $\min_i|Ef_{\tau_i}(x)|$ is essentially constant on each tube $B$, we get using (A3), \eqref{10} and the fact that $s\le4$
$$\int_{B_R}(\phi_\tau\min_i|Ef_{\tau_i}|)^s\approx\sum_B\int_B(\min_i|Ef_{\tau_i}|)^s\frac1{|B|}\int_B\phi_\tau^s$$
$$\l\int_{B_R}\min_i|Ef_{\tau_i}|^s\l \delta^{3s-5}.$$
Combining this with \eqref{11} leads to the following estimate for $q\le s\le 4$
\begin{equation}
\label{12}
\|\max_{\tilde{\E}_\delta}\;[\sum_{\tau\in \E_\delta}(\phi_\tau\min_i|Ef_{\tau_i}|)^2]^{1/2}\|_{L^s(B_R)}\l \delta^{\frac52-\frac7s}.
\end{equation}
We will later use this with $s=q$.
\bigskip

Here is how to refine this estimate. For dyadic parameters $0<\lambda\le 1$ and $\mu\ge 1$ write for each $\tilde{\tau}:=(\tau,\tau_1,\tau_2,\tau_3)\in\tilde{\E}_\delta$
$$g_{\tilde{\tau},\lambda}=\min_i|Ef_{\tau_i}|1_{\{\min_i|Ef_{\tau_i}|\sim \lambda\delta^3\}}$$
$$\phi_{\tau,\mu}=\phi_\tau1_{\phi_\tau\sim \mu},\;\;\mu>1$$
$$\phi_{\tau,1}=\phi_\tau1_{\phi_\tau\lesssim 1}.$$
Note that$$\min_i|Ef_{\tau_i}|=\sum_\lambda g_{\tilde{\tau},\lambda}$$
$$\phi_\tau=\sum_{\mu}\phi_{\tau,\mu}.$$
Because of the triangle inequality, it suffices to focus on fixed values of $\lambda,\mu$.
A repeat of the earlier argument using now
$$\frac1{|B|}\int_B\phi_{\tau,\mu}^s\l \mu^{s-4}$$
and
\begin{equation}
\label{16}
\int_{B_R}g_{\tilde{\tau},\lambda}^s\lesssim (\lambda \delta^3)^{s-q}\int_{B_R}\min_i|Ef_{\tau_i}|^q\l \lambda^{s-q}\delta^{3s-5}
\end{equation}
leads to the estimate for $q\le s\le 4$
$$
\|\max_{\tilde{\E}_\delta}\;[\sum_{\tau\in \E_\delta}(\phi_{\tau,\mu} g_{\tilde{\tau},\lambda})^2]^{1/2}\|_{L^s(B_R)}\l \lambda^{1-\frac{q}{s}}\mu^{1-\frac4s}\delta^{\frac52-\frac7s}.
$$

To simplify computations, we will later use the above with $s=\frac{14}{5}$
\begin{equation}\label{13}
\|\max_{\tilde{\E}_\delta}\;[\sum_{\tau\in \E_\delta}(\phi_{\tau,\mu} g_{\tilde{\tau},\lambda})^2]^{1/2}\|_{L^{\frac{14}{5}}(B_R)}\l \lambda^{1-\frac{5q}{14}}\mu^{-\frac37}.
\end{equation}
\bigskip

Let us now get the third estimate. Using our hypothesis, for each collection $\Omega$ consisting of $(\delta^{-1},\delta^{-2})$-tubes with $\delta$-separated directions we have
\begin{equation}
\label{14}
\|\sum_{T\in\Omega}1_T\|_{r}\l \delta^{-3-\frac4r}.
\end{equation}
A standard consequence via convexity is the estimate
\begin{equation}
\label{15}
\|\sum_{T\in\Omega}\sum_{k\in\Lambda_T}c_{T,k}1_{T+k}\|_{r}\l \delta^{-3-\frac4r},
\end{equation}
whenever $c_{T,k}\ge 0$ and $\max_{T\in\Omega}\sum_{k\in \Lambda_T}c_{T,k}\lesssim 1$, with $(T+k)_{k\in\Lambda_T}$ a tiling of $\R^4$.
\medskip

For each $\tau$ let $T_\tau$ be the tube dual to $\tau$ passing through the origin.  We can think of each $|Ef_{\tau_i}|$, and thus also of  $g_{\tilde{\tau},\lambda}^2$ as being essentially constant on each $k+T_\tau$, $k\in\Lambda_{T_\tau}$. More precisely
\begin{equation}
\label{18}
(g_{\tilde{\tau},\lambda}(x))^2\l \delta^5\int (g_{\tilde{\tau},\lambda}(z))^21_{T_\tau}(x-z)dz.
\end{equation}
A computation similar to \eqref{16} shows that
\begin{equation}
\label{17}
\int_{B_R}g_{\tilde{\tau},\lambda}^2\lesssim (\lambda \delta^3)^{2-q}\int_{B_R}\min_i|Ef_{\tau_i}|^q\l \lambda^{2-q}\delta.
\end{equation}
We can rewrite \eqref{18} and \eqref{17} as
$$(g_{\tilde{\tau},\lambda}(x))^2\l\delta^6\lambda^{2-q}\int1_{T_\tau}(x-z)c_{\tau,\lambda}(z)dz$$
with $c_{\tau,\lambda}$ essentially constant on each $k+T_{\tau}$ and satisfying
$$\int c_{\tau,\lambda}\l 1.$$
Note that for each $\tau$ there can be $\l 1$ many $\tilde{\tau}$ with first entry $\tau$. Thus
$$\max_{\tilde{\E}_\delta}\;[\sum_{\tau\in \E_\delta}(\phi_{\tau,\mu} g_{\tilde{\tau},\lambda}(x))^2]^{1/2}\lesssim \mu[\sum_{\tilde{\tau}}(g_{\tilde{\tau},\lambda}(x))^2]^{1/2}$$$$\l\delta^3\lambda^{1-\frac{q}2}[\int\sum_\tau 1_{T_\tau}(x-z)c_{\tau,\lambda}(z)dz]^{1/2}.$$
Invoking \eqref{15} we get our third main estimate
\begin{equation}
\label{19}
\|\max_{\tilde{\E}_\delta}\;[\sum_{\tau\in \E_\delta}(\phi_{\tau,\mu} g_{\tilde{\tau},\lambda})^2]^{1/2}\|_{L^{2r}(B_R)}\l \mu\lambda^{1-\frac{q}{2}}\delta^{\frac32-\frac2r}.
\end{equation}

Interpolate \eqref{13} and \eqref{19} as follows. Let $\theta\in(0,1)$ be such that
$$\theta(1-\frac{q}{2})+(1-\theta)(1-\frac{5q}{14})=0,$$
so $\theta=\frac{14-5q}{2q}$.
We may assume\footnote{Otherwise we actually have a stronger estimate and things get easier} $2r>\frac{14}{5}$. Define $p_1\in(\frac{14}{5},2r)$ via
$$\frac{1}{p_1}=\frac{\theta}{2r}+\frac{5(1-\theta)}{14}.$$
Then  \eqref{13} and \eqref{19} give
$$
\|\max_{\tilde{\E}_\delta}\;[\sum_{\tau\in \E_\delta}(\phi_{\tau,\mu} g_{\tilde{\tau},\lambda})^2]^{1/2}\|_{L^{p_1}(B_R)}\l \mu^{\frac{10\theta-3}{7}}\delta^{\theta f_1(\frac1{2r})+(1-\theta)f_2(\frac{5}{14})}$$

The choice of $\theta$ was made in order to make the exponent of $\lambda$ zero. Recall that $\mu\ge 1$. The key facts are  that $f_1(z)>f_2(z)$ for $z>\frac13$ and that $10\theta-3<0$. These together with the fact that $f_2$ is affine allows us to rewrite the above inequality as follows
$$\|\max_{\tilde{\E}_\delta}\;[\sum_{\tau\in \E_\delta}(\phi_{\tau,\mu} g_{\tilde{\tau},\lambda})^2]^{1/2}\|_{L^{p_1}(B_R)}\l \delta^{f_2(\frac{1}{p_1})+\Delta},$$
for some $\Delta>0$ whose exact value is not important. We may in fact choose a slightly larger $\theta$, so  that we have a saving in $\lambda$ that allows to sum over both $\mu\ge 1$ and $\lambda\le 1$. We conclude that
$$\|\max_{\tilde{\E}_\delta}\;[\sum_{\tau\in \E_\delta}(\phi_{\tau} \min_{i}|Ef_{\tau_i}|)^2]^{1/2}\|_{L^{p_1}(B_R)}\l \delta^{f_2(\frac{1}{p_1})+\Delta},$$
for some $p_1>\frac{14}{5}$. Interpolate this with \eqref{12} ($s=q$) which we rewrite as follows
$$
\|\max_{\tilde{\E}_\delta}\;[\sum_{\tau\in \E_\delta}(\phi_\tau\min_i|Ef_{\tau_i}|)^2]^{1/2}\|_{L^q(B_R)}\l \delta^{f_2(\frac{1}{q})}.
$$
Since $f_2(\frac{5}{14})=0$ and $f_2$ is affine, there is  $\alpha\in(0,1)$ so that
$$\alpha(f_2(\frac{1}{p_1})+\Delta)+(1-\alpha)f_2(\frac{1}{q})=0$$
and so that $p$ defined via
$$\frac{1}{p}=\frac{\alpha}{p_1}+\frac{1-\alpha}{q}$$
satisfies $p<\frac{14}{5}$. With this choice, H\"older leads to
$$
\|\max_{\tilde{\E}_\delta}\;[\sum_{\tau\in \E_\delta}(\phi_\tau\min_i|Ef_{\tau_i}|)^2]^{1/2}\|_{L^p(B_R)}\l 1.
$$
The desired inequality $\|\eqref{8}\|_{L^p(B_R)}\l 1$ is now immediate since there are $\l 1$ many scales $\delta$.

\end{document}